\documentclass[11pt,twoside]{article}

\usepackage[utf8]{inputenc}
\usepackage[T1]{fontenc}

\usepackage{latexsym}
\usepackage{amssymb,amsbsy,amsmath,amsfonts,amssymb,amscd,amsthm}
\usepackage{hyperref}
\usepackage{enumitem}
\usepackage{comment}
\usepackage{xcolor}
\usepackage{color}
\newcommand{\commentout}[1]{}
\newcommand{\R}{\mathbb{R}}
\newcommand{\N}{\mathbb{N}}

\newcommand{\1}{{\mathchoice {\rm 1\mskip-4mu l} {\rm 1\mskip-4mu l}
{\rm 1\mskip-4.5mu l} {\rm 1\mskip-5mu l}}}
\newcommand{\ep}{\epsilon}
\newcommand {\al} {\alpha}
\newcommand {\e}  {\epsilon}
\newcommand {\eps}  {\epsilon}

\newcommand {\Chi} {{\bf \raise 2pt \hbox{$\chi$}} }
\renewcommand {\r} { {\mathbf r} }
\newcommand {\n} { {\mathbf n} }
\newcommand {\f}   {\frac}
\newcommand {\p}   {\partial}

\newcommand{\dis}{\displaystyle}
\newcommand{\beq}{\begin{equation}}
\newcommand{\beqa} {\begin{array}{rl}}
\newcommand{\eeq}{\end{equation}}
\newcommand{\eeqa}{\end{array}}
\newtheorem{theorem}{Theorem}
\newtheorem{lemma}[theorem]{Lemma}
\newtheorem{definition}[theorem]{Definition}
\newtheorem{remark}[theorem]{Remark}
\newtheorem{proposition}[theorem]{Proposition}
\newtheorem{corollary}[theorem]{Corollary}

\title{\Large \bf Subdiffusive fractional limit of a jump-renewal equation}

\author{
Hugues Berry\thanks{AIstroSight, Inria, Hospices Civils de Lyon, Université Claude Bernard Lyon 1, 56 Bld Niels Bohr, Villeurbanne, 69603, France; email: hugues.berry@inria.fr}
\and
Pierre Gabriel\thanks{Université de Tours, Université d’Orléans, CNRS, IDP, UMR 7013, Tours, France; email: pierre.gabriel@univ-tours.fr}
\and
Thomas Lepoutre\thanks{Inria, CNRS, Centrale Lyon, INSA Lyon, Universite Claude Bernard Lyon 1, Université Jean Monnet, I CJ UMR5208, 69603
Villeurbanne, France; email: thomas.lepoutre@inria.fr}
\and
Nathan Quiblier\thanks{Université Paris Cité,
Institut Imagine, INSERM UMR 1163, Paris, France; email: nathan.quiblier@institutimagine.org}
}

\date{}

\begin{document}
\maketitle
\pagestyle{plain}
\pagenumbering{arabic}

\begin{abstract}
In this paper, we consider an age-structured jump model that arises as a description of continuous time random walks with infinite mean waiting time between jumps.
We prove that under a suitable rescaling, this equation converges in the long time large scale limit to a time fractional subdiffusion equation.
\end{abstract}

\medskip



\section{Introduction}
A distinctive feature of Brownian motion is that the variance of the travelled distance, or mean squared displacement (MSD), scales linearly with time. However, accumulated experimental recordings of the motion of biomolecules in living cells evidence random motions where the MSD grows sublinearly, as a powerlaw in time, say as $t^\al$ with $0<\al<1$.
Such a nonlinear growth of the MSD is usually referred to as anomalous diffusion, and more precisely as subdiffusion in the case of a sublinear growth. Subdiffusion has been experimentally reported in various biological systems, from actin networks \cite{amblard_subdiffusion_1996} to active intracellular transport \cite{caspi_enhanced_2000}, or in single-molecule tracking experiments of transcription factors in the nucleus \cite{izeddin_single-molecule_2014}. Applications of subdiffusion are wide, ranging from the modelling of protein diffusion in crowded biological media \cite{Saxton2007}, medical and drug delivery systems \cite{Metzler2020}, viscoelastic behaviors in the nucleus \cite{Lee2021}, subdiffusion-controlled chemical reactions \cite{Haugh2009}, or more generally, diffusion in porous media (see e.g., \cite{Graczyk2023}).

A very abundant literature has proposed mathematical models able to account for subdiffusion in biology. We refer firstly to several textbooks (and the references therein), namely, two introductory books \cite{evangelista2018fractional,evangelista2023introduction} or a more cell-oriented modelling point of view in \cite[chap. 6-7]{bressloff2022stochastic}. A variety of points of view has been expressed regarding what models may be at the origin of the experimental records showing anomalous diffusion \cite[part III]{bressloff2013stochastic}; the interested reader will find a collection of applicative examples in \cite{hofling2013,suzuki2023fractional,wang2022anomalous}. 
Historically, the first tool proposed to model mathematically such behaviors has been the so called continuous time random walk (CTRW)~\cite{MontrollWeiss1965},
a generalization of random walk where the duration between two random successive jumps, during which the walker maintains his position, is modelled as a random variable with associated distribution.
Assuming separability, such a process is therefore characterized by two probability distributions, one for the jumps in the position space, typically a kernel $w$ on $\R^d$,
and a second one for the waiting or residence time, that we will call $\phi$, on the time half line $[0,\infty)$.

\medskip

We denote by $r(t,a,x)$ the density of particles, molecules, or individuals, which, at time $t$, have been resident at position $x$  for the last $a$ time units. This density distribution $r$ satisfies the following nonlocal partial differential equation
\begin{equation}\label{eq:jump-renewal}
\left\{\begin{array}{l}
\dis\f{\p}{\p t} r(t,a,x)+\f{\p}{\p a} r(t,a,x)+\beta(a)r(t,a,x)=0,\qquad\qquad t,a>0,\quad x\in\R^d,
\vspace{2mm}\\
\dis \r(t,x):=r(t,0,x)=\int_0^\infty\!\!\!\int_{\R^d}w(x-y)\beta(a)r(t,a,y)\,dyda,
\vspace{2mm}\\
r(0,a,x)=r^0(a,x),
\end{array}\right.
\end{equation}
where $\beta$ is a jump rate, or escape rate.

In opposition to age-based models for epidemics, for instance, that usually employ age-dependent death rates and new individual births with age 0, the particles in our model do not die or birth, they move through jumps in space with an age-dependent jump rate $\beta(a)$. The age variable $a$ does not represent a biological age, but the time spent since the last jump of the particle, so that any jump resets the particle age to zero, thus the r.h.s. of the boundary term $\r(t,x)$. Therefore, in this setting, a jump at age $a$ corresponds to a loss of individuals of age $a$ and to a (balanced) gain of individual of age $0$, which is basically what eq.\eqref{eq:jump-renewal} expresses.

Upon arriving at a new position at time $0$, the probability of not jumping before time $t$, sometimes called the {\it survival function}, is given by
\[\psi(t)=e^{-\int_0^t\beta(s)ds}.\]
The density probability function of the times of jumps is then given by
\[\phi(t)=\beta(t)\psi(t)=\beta(t)\,e^{-\int_0^t\beta(s)ds}=-\psi'(t).\]

This model falls into the class of age-structured models, since the residence time $a$ behaves like an ``age'' for the particles, which is by definition reset to zero when a jump occurs.
A similar equation appears in the context of kinetic models in domains with holes in~\cite{Bernard2010}.
For more on age-structured modelling, we refer to the classical textbook \cite{perthame2007transport}.

Equation~\eqref{eq:jump-renewal} was first proposed by~\cite{Vlad1984} as an alternative to the generalized master equation (GME) which is often used to describe CTRW, see for instance, among many others, \cite{Gorenflo2008,Metzler2000,Scalas2004}.
As noticed in~\cite{Vlad1987}, unlike the GME which does not incorporate age variable, Equation~\eqref{eq:jump-renewal} corresponds to a Markovian description of the random walk.
At the mesoscopic level it means that Equation~\eqref{eq:jump-renewal} enjoys a semigroup property.
This makes it more relevant than GME in particular for incorporating reaction terms, see~\cite{Vlad2002,Yadav2006,Mendez2010}.

\medskip

We suppose that $w$ is a probability measure on $\R^d$ with finite second moment and zero first moment, namely
\begin{equation}\label{as:w}
w\geq0,\qquad\int_{\R^d}w(dx)=1,\qquad\int_{\R^d}xw(dx)=0,\qquad \sigma^2:=\int_{\R^d}|x|^2w(dx)<\infty.
\end{equation}
We will also assume that $w$ satisfies the isotropy condition
\begin{equation}\label{as:w2}
\forall i\in\{1,\cdots,d\},\quad\int_{\R^d}x_i^2w(dx)=\frac{\sigma^2}{d}.
\end{equation}
The mean waiting time of a trapped particle before escaping is given by
\[T:=\int_0^\infty t\phi(t)dt=\int_0^\infty\psi(t)dt.\]
When this quantity is infinite, the underlying random process is expected to exhibit a subdiffusive behaviour, and diffusive when it is finite (recall that we are in the case $\sigma^2<\infty$), see for instance~\cite{Metzler2000}.
Here we are interested in the subdiffusive case and we will then assume that $T=+\infty$.
More precisely we suppose that $\beta$ is a bounded continuous function on $[0,\infty)$ such that
\begin{equation}\label{as:beta}
    \psi(t)=\Psi t^{-\al} + O(t^{-\al-\delta})\qquad\text{as}\quad t\to+\infty
\end{equation}
for some constants $\Psi>0$, $\al\in(0,1)$ and $\delta\in(0,1-\al)$.
A prototypical example is provided by the jump rate functions $\beta(a)=\al/(K+a)$, with $K>0$,
for which $\psi(t)=1/(1+t/K)^\al$.

\medskip

With the above coefficients, we have the following results about the behavior of the solutions to Equation~\eqref{eq:jump-renewal},
and more precisely about the behavior of the marginal
\[\rho(t,x):=\int_0^\infty r(t,a,x)\,da.\]

\begin{proposition}\label{prop:jump-renewal}
Under Assumptions~\eqref{as:w} and~\eqref{as:beta}, the solutions to Equation~\eqref{eq:jump-renewal} satisfy
\begin{equation}\label{eq:rhomom0}
\int_{\R^d}\rho(t,x)\,dx=\int_{\R^d}\rho^0(x)\,dx,\qquad \forall t\geq0,
\end{equation}
\begin{equation}\label{eq:rhomom1}
\int_{\R^d}x\rho(t,x)\,dx=\int_{\R^d}x\rho^0(x)\,dx,\qquad \forall t\geq0,
\end{equation}

\smallskip

and, if $\int_{\R^d}|x|^2\rho^0(x)\,dx<\infty$,
\begin{equation}\label{eq:rhomom2}
\int_{\R^d}|x|^2\rho(t,x)\,dx\sim\frac{\sin(\pi\al)}{\pi\al}\frac{\sigma^2}{\Psi}\bigg(\int_{\R^d}\rho^0(x)dx\bigg) t^\al,\qquad \text{as}\ t\to\infty.
\end{equation}
\end{proposition}

\medskip

The result in eq.~\eqref{eq:rhomom2} emphasizes the fact that the MSD of the underlying random walk grows sublinearly (more precisely as $t^\al$), thus showing that the long time limit of the age-structured model eq.~\eqref{eq:jump-renewal} exhibits subdiffusion.

This result also motivates performing a time-space rescaling $(t,a,x)\to (t/\e^{2/\al},a,x/\e)$.
More precisely we consider the equation
\begin{equation}\label{eq:jump-renewal-eps}
\left\{\begin{array}{l}
\dis \epsilon^{2/\alpha}\f{\p}{\p t} r_\epsilon(t,a,x)+\f{\p}{\p a} r_\epsilon(t,a,x)+\beta(a)r_\epsilon(t,a,x)=0,\qquad\qquad t,a>0,\quad x\in\R^d,
\vspace{2mm}\\
\dis \r_\e (t,x) := r_\epsilon(t,0,x)=\int_0^\infty\!\!\!\int_{\R^d}w_\e(x-y)\beta(a)r_\epsilon(t,a,y)\,dyda,
\vspace{2mm}\\
r_\e(0,a,x)=r_\e^0(a,x),
\end{array}\right.
\end{equation}
where we have set
\[w_\e(z)=\e^{-d}\, w(z/\e).\]

\medskip

For a survival function $\psi$ satisfying~\eqref{as:beta} and a jump kernel satisfying~\eqref{as:w} and~\eqref{as:w2}, it was obtained by~\cite{Yadav2006}, see also~\cite[Chap.~2.3]{Mendez2010}, through a formal derivation based on asymptotic expansion in the Laplace (in time) -- Fourier (in space) domain, that the dynamics of the marginal $\rho$ of the solutions to Equation~\eqref{eq:jump-renewal}
is prescribed in the large scale and large time regime by the time fractional heat equation
\begin{equation}\label{eq:frac-subdiff}
\partial_t^\al \rho(t,x)=D_\al\,\Delta\rho(t,x).
\end{equation}
In this equation, $D_\al>0$ is a (sub)diffusion coefficient and $\partial_t^\alpha$ stands for the Caputo fractional derivative defined for $f\in C^1([0,\infty))$ by
\[\partial_t^\alpha f(t)=\frac{1}{\Gamma(1-\al)}\int_0^t\frac{f'(s)}{(t-s)^\al}ds\]
with $\Gamma(\al)=\int_0^\infty t^{\al-1}e^{-t}dt$ the Gamma function.
In the last years, many mathematical properties of such time fractional parabolic equations have been obtained, see for instance~\cite{DongKim,Zacher2016,Zacher2013}.

\medskip

The main objective of the present paper is to derive rigorously Equation~\eqref{eq:frac-subdiff} as the limit of Equation~\eqref{eq:jump-renewal-eps} when $\e\to0$.
Prior to that, we need to give a definition of what we call a solution to Equation~\eqref{eq:frac-subdiff}.
It can be shown, see Corollary~\ref{cor:fracODE} in Section~\ref{sec:prelim}, that for $f\in C^1([0,\infty))$ and $g\in C([0,\infty))$, the equation
\[\partial_t^\al f=g\]
is equivalent to the integrated form
\[f(t)=f(0)+\frac{1}{\Gamma(\al)}\int_0^t\frac{g(s)}{(t-s)^{1-\al}}ds.\]
Noticing that the second formulation requires less regularity on $f$, this motivates our definition in a mild sense of solution to Equation~\eqref{eq:frac-subdiff}.
Concerning the regularity in space, we only require $\rho(t,\cdot)$ to be a finite positive measure.
We then consider $\mathcal M(\R^d)=(C_0(\R^d))'$ the space of finite signed measures that we endow with its weak-$*$ topology, namely
\[\rho_n\rightharpoonup\rho\quad\text{in}\ \mathcal M(\R^d)\qquad\Longleftrightarrow\qquad\forall \varphi\in C_0(\R^d),\quad \int_{\R^d}\varphi(x)\rho_n(dx)\to\int_{\R^d}\varphi(x)\rho(dx),\]
where $C_0(\R^d)$ is the set of continuous functions on $\R^d$ that tend to zero at infinity.
The positive cone of $\mathcal M(\R^d)$, which is made of the finite positive measures on $\R^d$, is denoted as $\mathcal M_+(\R^d)$.

\begin{definition}\label{def:sol}
We say that $\rho\in C([0,\infty),\mathcal M(\R^d))$ is a solution to Equation~\eqref{eq:frac-subdiff} with initial condition $\rho(0)=\rho^0\in\mathcal M(\R^d)$ if for any $\varphi\in C^2_c(\R^d)$ and all $t\geq0$
\begin{equation}\label{eq:sol}
\int_{\R^d}\varphi(x)\rho(t,dx)=\int_{\R^d}\varphi(x)\rho^0(dx)
+\frac{D_\al}{\Gamma(\al)}\int_0^t\frac{1}{(t-s)^{1-\al}}\int_{\R^d}\Delta\varphi(x)\rho(s,dx)\,ds.
\end{equation}
\end{definition}

We now state the main result of the paper.

\begin{theorem}\label{th:main}
Suppose that Assumptions~\eqref{as:w},~\eqref{as:w2} and~\eqref{as:beta} hold, that
\begin{equation}\label{as:r0}
r_\e^0\geq0\qquad\text{and}\qquad \int_{\R^d}\int_0^\infty r_\e^0(a,x)(1+a)^\al dadx\leq M
\end{equation}
for some $M>0$ and all $\e\in(0,1)$, and that
\begin{equation}\label{eq:conv-initial}
\int_{\R^d}\int_0^\infty r_\e^0(a,x)\varphi(x)\,dadx\to\int_{\R^d}\varphi(x)\rho^0(dx)\qquad\text{as}\ \e\to0
\end{equation}
for some $\rho^0\in\mathcal M_+(\R^d)$ and all $\varphi\in C_0(\R^d)$.

\smallskip

Then there exists a sequence ${(\e_n)}_{n\geq1}$ decreasing towards zero and a solution $\rho\in C([0,\infty),\mathcal M(\R^d))$ to Equation~\eqref{eq:frac-subdiff} with initial condition $\rho(0)=\rho^0$, in the sense of Definition~\ref{def:sol}, with
\[D_\al=\frac{\sigma^2}{2d\Psi\Gamma(1-\al)}\]
such that for all $T>0$, the marginal $\rho_{\e_n}$ of the solution $r_{\e_n}$ to Equation~\eqref{eq:jump-renewal-eps} with $\e=\e_n$ satisfies
\[\rho_{\e_n}\to\rho\qquad\text{in}\ C([0,T],\mathcal M(\R^d))\ \text{as}\ n\to\infty.\]
\end{theorem}

\medskip

Let us make a few comments on this theorem.

\smallskip

Typical examples of initial distributions that verify~\eqref{eq:conv-initial} are:

- the case of well prepared initial conditions where $r_\e^0(a,x)=r^0(a,x)$ is independent of $\e$,

- the case of rescaling Equation~\eqref{eq:jump-renewal} with $r_\e(t,a,x)=\e^{-d}r(t/\e^{2/\al},a,x/\e)$ which satisfies~\eqref{eq:jump-renewal-eps}  with initial condition $r_\e^0(a,x)=\e^{-d}r^0(a,x/\e)$ and verifies the convergence~\eqref{eq:conv-initial} with $\rho^0=\delta_0$.

\smallskip

A similar result as Theorem~\ref{th:main} was obtained recently in~\cite{Perthame2025} by means of Laplace transform in time.
Contrary to their approach, our method consists in deriving estimates in the physical domain.
In particular our result does not rely on the injectivity of the Laplace transform.

\smallskip

A different scaling of Equation~\eqref{eq:jump-renewal} was performed in~\cite{Calvez2019} together with a Hopf-Cole transformation, yielding a Hamilton-Jacobi limiting equation.
The Hamiltonian encodes the subdiffusive character of the equation through its behavior at the origin, which is a powerlaw with power $2/\al$.

\smallskip

The last comment is about the solutions of Equation~\eqref{eq:frac-subdiff} in the sense of Definition~\ref{def:sol}, which satisfy the same moments estimates as for the solutions of Equation~\eqref{eq:jump-renewal}.

\begin{proposition}\label{prop:frac-subdiff}
The solutions of Equation~\eqref{eq:frac-subdiff} satisfy~\eqref{eq:rhomom0}, \eqref{eq:rhomom1}, and
\begin{equation}\label{eq:rhomom2bis}
\int_{\R^d}|x|^2\rho(t,dx)=\int_{\R^d}|x|^2\rho^0(dx)+\frac{\sin(\pi\al)}{\pi\al}\frac{\sigma^2}{\Psi}\bigg(\int_{\R^d}\rho^0(dx)\bigg) t^\al,\qquad \forall t\geq0.
\end{equation}
\end{proposition}

\smallskip

In the next section, we provide preliminary results that will be useful for proving Theorem~\ref{th:main} and we give the proofs of Propositions~\ref{prop:jump-renewal} and~\ref{prop:frac-subdiff}.
The last section is devoted to the proof of Theorem~\ref{th:main}.

\section{Preliminary results and proofs of the propositions}
\label{sec:prelim}

Integrating Equation~\eqref{eq:jump-renewal} in the space variable, we readily obtain that the quantity
\[n(t,a):=\int_{\R^d} r(t,a,x)\,dx\]
verifies the equation
\begin{equation}\label{eq:age_density}
    \begin{cases}
    \partial_t n+\partial_a n+\beta(a)n=0,
    \vspace{2mm}\\
    \dis \n(t):=n(t,0)=\int_{0}^\infty\beta(a)n(t,a)da,
    \vspace{2mm}\\
    n(0,a)=n^0(a).
    \end{cases}
\end{equation}
This equation is sometimes known as the conservative renewal equation, in the sense that the integral $\int_0^\infty n(t,a)da$ is preserved along time.
It has been widely studied in the literature.
Most of the studies are in the case where $\lim\inf_{a\to\infty}\beta(a)>0$, see for instance~\cite{Gabriel2018} and the references therein.
For the case where $\beta$ tends to zero as $\al/a$ when $a\to\infty$, which is the case we are considering here, see~\cite{Berry2016} and the references therein.

\medskip

For $f,g\in L^1_{loc}(0,\infty)$, we define the convolution $f*g\in L^1_{loc}(0,\infty)$ by
\[f*g(t)=\int_0^t f(s)g(t-s)\,ds=\int_0^t f(t-s)g(s)\,ds.\]
We recall the useful commutativity and associativity properties  $f*g=g*f$, $(f*g)*h=f*(g*h)$.

We also define, for any $\nu>0$, the function $Y_\nu\in L^1_{loc}(0,\infty)$ by
\[Y_\nu(t)=\frac{t^{\nu-1}}{\Gamma(\nu)},\]
so that for $\al\in(0,1)$ we have
\[\partial_t^\al f(t)=Y_{1-\al}*f'(t).\]
We recall a classical result about the convolution of such functions $Y_\nu$.

\begin{lemma}\label{lem:Yalpha}
For any $\mu,\nu>0$ one has
\[Y_\mu*Y_\nu=Y_{\mu+\nu}.\]
In particular, $Y_{\alpha}*Y_{1-\alpha}=Y_1=1$.
\end{lemma}

\begin{proof}
It is based on the standard formula
\begin{equation}\label{eq:loiBeta}
\int_0^1(1-s)^{\mu-1}s^{\nu-1}ds=\frac{\Gamma(\mu)\Gamma(\nu)}{\Gamma(\mu+\nu)}.
\end{equation}
For the sake of completeness we recall the classical proof of this formula through a change of variable.
Starting from
\[\Gamma(\mu)\Gamma(\nu)=\int_0^\infty u^{\mu-1}e^{-u}du\int_0^\infty v^{\nu-1}e^{-v}dv\]
and using the change of variable $(u,v)=(sw,(1-s)w)$, the function $\varphi:(0,1)\times(0,\infty)\to(0,\infty)^2$ defined by $\varphi(s,w)=((1-s)w,sw)$ being a bijection with Jacobian determinant equal to $-w$, we get that
\begin{align*}
\Gamma(\mu)\Gamma(\nu)&=\int_0^\infty\int_0^1(1-s)^{\mu-1}w^{\mu-1}e^{-(1-s)w}(sw)^{\nu-1}e^{-sw}w\,dsdw\\
&=\int_0^\infty w^{\mu+\nu-1}e^{-w}dw\int_0^1 (1-s)^{\mu-1}s^{\nu-1}ds,
\end{align*}
and the result.
Now for $t>0$ we have by using the change of variable $s\to ts$ that 
\[\int_0^t(t-s)^{\mu-1}s^{\nu-1}ds=t^{\mu+\nu-1}\int_0^1(1-s)^{\mu-1}s^{\nu-1}ds\]
which readily gives by using~\eqref{eq:loiBeta}
\[Y_\mu*Y_\nu(t)=Y_{\mu+\nu}(t).\]
\end{proof}

\begin{corollary}\label{cor:fracODE}
For $f\in C^1([0,\infty),\R)$ and $g\in C([0,\infty),\R)$ we have
\[\bigg[\ \forall t\geq0,\quad\partial_t^\al f(t)=g(t)\ \bigg]\quad\Longleftrightarrow\quad\bigg[\ \forall t\geq0,\quad f(t)=f(0)+\frac{1}{\Gamma(\al)}\int_0^t\frac{g(s)}{(t-s)^{1-\al}}ds\ \bigg].\]
\end{corollary}

\begin{proof}
The equivalence reads
\[Y_{1-\al}*f'=g\quad\Longleftrightarrow\quad f-f(0)=Y_\al*g.\]
From the left to the right, it suffices to convolve by $Y_\al$, use Lemma~\ref{lem:Yalpha} and note that $Y_1*f'(t)=f(t)-f(0)$.
For going from the right to the left, a convolution with $Y_{1-\al}$ first yields
\[\frac{1}{\Gamma(1-\al)}\int_0^t\frac{f(t-s)-f(0)}{s^\al}ds=\int_0^t g(s)\,ds,\]
and a differentiation in $t$ then gives the result.
\end{proof}

We now give useful estimates on the convolution of the boundary condition $\n$ of Equation~\eqref{eq:age_density} with $Y_\nu$.
We start with a result on the convolution of $\n$ with $\psi$, which extends the result in~\cite[Prop.~3]{Perthame2025}.

\begin{lemma}\label{lem:renewal}
Assume that $\psi(t)\searrow 0$ as $t\to\infty$, which is equivalent to considering $\beta\not\in L^1(0,\infty)$.
Then the boundary condition $\n(t)$ of Equation~\eqref{eq:age_density} satisfies
\[\n*\psi(t)\to \int_0^\infty n^0(a)\,da\qquad\text{as}\quad t\to+\infty.\]
\end{lemma}

\begin{proof}
Using the method of characteristics, we have
\[n(t,a)=\mathbf n(t-a)\psi(a)\1_{t\geq a}+n^0(a-t)\frac{\psi(a)}{\psi(a-t)}\1_{t<a}.\]
Injecting this into the boundary condition we get
\[n(t,0)=\int_0^\infty\beta(a)n(t,a)\,da=\int_0^t\phi(a)\n(t-a)\,da+\int_t^\infty\frac{\phi(a)}{\psi(a-t)}n^0(a-t)\,da,\]
which also reads
\begin{equation}\label{eq:N1}
\n(t)=\phi*\n(t)+\int_0^\infty\frac{\phi(a+t)}{\psi(a)}n^0(a)\,da.
\end{equation}
Noticing that
\[Y_1*\phi=1-\psi=Y_1-\psi\]
we deduce from~\eqref{eq:N1} that
\begin{align*}
\psi*\n(t)=Y_1*(\n-\phi*\n)(t)&=\int_0^t\int_0^\infty\frac{\phi(a+s)}{\psi(a)}n^0(a)\,dads\\
&=\int_0^\infty \frac{n^0(a)}{\psi(a)}\int_0^t\phi(a+s)\,dsda\\
&=\int_0^\infty \frac{n^0(a)}{\psi(a)}\big[\psi(a)-\psi(a+t)\big]\,da\\
&=\int_0^\infty n^0(a)\,da - \int_0^\infty \frac{\psi(a+t)}{\psi(a)}n^0(a)\,da\to \int_0^\infty n^0(a)\,da,
\end{align*}
where we have used the dominated convergence theorem for the last convergence and where $n^0(a)$ is an integrable dominating function.
Indeed, since $\psi$ is a non-increasing function that goes to zero at infinity, we have $\psi(a+t)/\psi(a)\leq1$ and $\psi(a+t)/\psi(a)\to0$ as $t\to\infty$ for any $a\geq0$. This proves pointwise convergence.
\end{proof}

\begin{remark}\label{rk:lem:renewal}
Note that the assumption on $\psi$ in Lemma~\ref{lem:renewal} is weaker than~\eqref{as:beta}.

However, under Assumption~\eqref{as:beta} and supposing besides, in the spirit of~\eqref{as:r0}, that
\begin{equation}\label{as:n0}
\int_0^\infty |n^0(a)|(1+a)^\al da\leq M,
\end{equation}
the convergence can be quantified.
Indeed, \eqref{as:beta} ensures the existence of $C>0$ such that
\[
\frac{\psi(a+t)}{\psi(a)}\leq C \left(\frac{1+a}{1+a+t}\right)^\alpha,
\]
so that 
 \[
 \bigg|\psi*\n(t)-\int n^0\bigg|\leq \int_0^\infty \frac{\psi(a+t)}{\psi(a)}|n^0(a)|\,da \leq CM 2^\al (1+t)^{-\al}.
 \]
\end{remark}

\bigskip

\begin{corollary}\label{cor:convol}
Assume~\eqref{as:beta}.
Then for any $\mu>1-\al$ we have
\[\n*Y_{\mu}(t)\sim \frac{\int_0^\infty n^0(a)\,da}{\Psi\Gamma(1-\al)}\, Y_{\mu+\al}(t)\qquad\text{as}\ t\to+\infty.\]
\end{corollary}

\begin{proof}
Since Equation~\eqref{eq:age_density} is linear, we can suppose without loss of generality that $n^0\geq0$, so that $n(t,\cdot)\geq0$ for all $t\geq0$.

We start by computing, for all $\nu>0$,
\begin{align*}
  Y_\nu*\psi(t)&=\frac{1}{\Gamma(\nu)}\int_0^t(t-s)^{\nu-1}\psi(s)\,ds\\
&=\frac{t^\nu}{\Gamma(\nu)}\int_0^1(1-s)^{\nu-1}\psi(ts)\,ds\\  
&= \frac{t^{\nu-\alpha}}{\Gamma(\nu)}\int_0^1(1-s)^{\nu-1}s^{-\alpha} (\Psi+(ts)^\alpha\psi(ts)-\Psi)\,ds.
\end{align*}
Since Assumption~\eqref{as:beta} ensures the existence of $C>0$ such that $|(ts)^\alpha\psi(ts)-\Psi|\leq C(1+ts)^{-\delta}$,
we infer, using~\eqref{eq:loiBeta} and the fact that $(1+ts)^{-\delta}\leq (1+t)^{-\delta}s^{-\delta}$ when $s\in(0,1)$, that
\begin{equation}\label{eq:Ynupsi}
Y_\nu*\psi(t)=\Psi\Gamma(1-\al)Y_{1+\nu-\al}(t)+\tau_\nu(t)t^{\nu-\alpha}
\end{equation}
with
\[|\tau_\nu(t)|\leq C\frac{\Gamma(1-\al-\delta)}{\Gamma(1+\nu-\al-\delta)}(1+t)^{-\delta}\xrightarrow[t\to\infty]{}0.\]
Then we have 
\begin{align*}
    Y_{1+\nu-\al}=\frac{1}{\Psi\Gamma(1-\al)}Y_\nu*\psi - \frac{\Gamma(1+\nu-\al)}{\Psi\Gamma(1-\al)}\, \tau_\nu Y_{1+\nu-\al}.
\end{align*}
We make the convolution with $\n$ and get 
\begin{align*}
    Y_{1+\nu-\al}*\n(t)=\frac{1}{\Psi\Gamma(1-\al)}Y_\nu*\psi*\n(t) - \frac{\Gamma(1+\nu-\al)}{\Psi\Gamma(1-\al)} (\tau_\nu Y_{1+\nu-\al})*\n(t).
\end{align*}
In the right hand side, since $\nu>0$ and $\n*\psi \rightarrow \int n^0 $ from Lemma~\ref{lem:renewal}, we have firstly
\begin{align*}
\frac{1}{\Psi\Gamma(1-\al)}Y_\nu*\psi*\n(t)&=  \frac{1}{\Psi\Gamma(1-\al)}\frac{t^\nu}{\Gamma(\nu)}\int_0^1 (1-s)^{\nu-1}(\psi*\n)(ts)\,ds \\
&\sim  \frac{1}{\Psi\Gamma(1-\al)}\frac{t^\nu}{\Gamma(\nu)}\left(\int n^0\right) \frac{1}{\nu}=\frac{1}{\Psi\Gamma(1-\al)}\left(\int n^0\right)Y_{1+\nu}(t).
\end{align*}
For the second term we write, for some $\eta\in(0,1)$ to be chosen later,
\begin{align*}
    \Gamma(1+\nu-\al)\,(\tau_\nu Y_{1+\nu-\al})*\n(t)&= \int_0^t \tau_\nu(s)s^{\nu-\al}\n(t-s)\,ds \\
    &= \int_0^{t^\eta} \tau_\nu(s)s^{\nu-\al}\n(t-s)\,ds+\int_{t^\eta}^t  \tau_\nu(s)s^{\nu-\al}\n(t-s)\,ds=I_1+I_2
\end{align*}
Since Equation~\eqref{eq:age_density} preserves the integral of the solutions along time and since $\beta$ is positive and bounded by assumption, we have
\[0\leq\n(t)\leq{\|\beta\|}_\infty \int_0^\infty n^0(a)\,da\]
and we can thus bound the first integral by 
\[
|I_1|\leq\int_0^{t^\eta} |\tau_\nu(s)s^{\nu-\al}\n(t-s)|\,ds\leq {\|\tau_\nu\|}_\infty {\|\n\|}_\infty \frac{(t^\eta)^{1+\nu-\al}}{1+\nu-\al}.
\]
Choosing  $\eta>0$ small enough so that $\eta(1+\nu-\al)<\nu $, we get that $I_1=o(t^{\nu})=o(Y_{1+\nu})$.
The second integral can be controlled by 
\begin{align*}
|I_2|&\leq \int_{t^\eta}^t  |\tau_\nu(s)s^{\nu-\al}\n(t-s)|\,ds\\
&\leq C\frac{\Gamma(1-\al-\delta)}{\Gamma(1+\nu-\al-\delta)}(1+t^\eta)^{-\delta}\int_{t^\eta}^t  s^{\nu-\al}\n(t-s)\,ds\\
 &\leq C\frac{\Gamma(1-\al-\delta)\Gamma(1+\nu-\al)}{\Gamma(1+\nu-\al-\delta)}  (1+t^\eta)^{-\delta} Y_{1+\nu-\al}*\n(t),
\end{align*}
so that $I_2=o(Y_{1+\nu-\al}*\n)$.
Combining everything together, 
\[
Y_{1+\nu-\al}*\n(t)= \frac{1}{\Psi\Gamma(1-\al)}\left(\int n^0\right)Y_{1+\nu}(t)\left(1+o(1)\right)+o(Y_{1+\nu}(t))+o(Y_{1+\nu-\al}*\n(t)),
\]
which implies the desired equivalence by taking $\nu=\mu+\al-1$.
\end{proof}

\begin{remark}\label{rk:cor:convol}
Recalling Remark~\ref{rk:lem:renewal} and looking carefully at above proof,
we can see that under Assumption~\eqref{as:n0} the result of Corollary~\ref{cor:convol} can be made more precise with the estimate
\[\n*Y_{\mu}(t)=\frac{\int_0^\infty n^0(a)\,da}{\Psi\Gamma(1-\al)}\, Y_{\mu+\al}(t)+O(t^{-\al}+t^{\eta\mu}+t^{-\eta\delta})\qquad\text{as}\ t\to\infty\]
for any choice of $\eta$ such that $0<\eta<1-\frac{1-\al}{\mu}$.
\end{remark}

Corollary~\ref{cor:convol} will be crucial in our proof of Theorem~\ref{th:main}, but it also allows us to prove Proposition~\ref{prop:jump-renewal}.

\begin{proof}[Proof of Proposition~\ref{prop:jump-renewal}]
The first two identities~\eqref{eq:rhomom0} and~\eqref{eq:rhomom1} are directly obtained by integration of Equation~\eqref{eq:jump-renewal}, by using that $\int w=1$ and $\int xw=0$ respectively.
For the equivalence of the second moment in~\eqref{eq:rhomom2}, we first show that
\begin{equation}\label{eq:rsecondmoment}
\frac{d}{dt}\int_0^\infty\!\!\!\int_{\R^d}|x|^2r(t,a,x)\,dxda=\sigma^2\,\n(t).
\end{equation}
Using again that $\int w=1$, we have by integration of Equation~\eqref{eq:jump-renewal} multiplied by $|x|^2$ that
\begin{align*}
\frac{d}{dt}\int_0^\infty\!\!\!\int_{\R^d}|x|^2r(t,a,x)\,dxda
& = \int_0^\infty\!\!\!\int_{\R^d}|x|^2\frac{\partial}{\partial t}r(t,a,x)\,dxda\\
& = -\int_0^\infty\!\!\!\int_{\R^d}|x|^2\Big(\frac{\partial}{\partial a}r(t,a,x)+\beta(a)r(t,a,x)\Big)\,dxda\\
&=\int_{\R^d}|x|^2\r(t,x)\,dx-\int_0^\infty\!\!\!\int_{\R^d}|x|^2\beta(a)r(t,a,x)\,dxda\\
&=\int_0^\infty\!\!\!\iint_{\R^{2d}}w(x-y)|x|^2\beta(a)r(t,a,y)\,dydxda\\
&\qquad\qquad -\int_0^\infty\!\!\!\iint_{\R^{2d}}w(y-x)|x|^2\beta(a)r(t,a,x)\,dydxda\\
&=\int_0^\infty\beta(a)\int_{\R^d}r(t,a,x)\bigg(\int_{\R^d}w(y-x)(|y|^2-|x|^2)\,dy\bigg)\,dxda.
\end{align*}
Using again~\eqref{as:w}, we have
\begin{align}\label{eq:wsigma}
\int_{\R^d}w(y-x)(|y|^2-|x|^2)\,dy&=\int_{\R^d}w(z)(|x+z|^2-|x|^2)\,dz\nonumber\\
&=\int_{\R^d}|z|^2w(z)\,dz+2x\cdot\int_{\R^d}zw(z)\,dz=\sigma^2
\end{align}
and thus~\eqref{eq:rsecondmoment}.
Integrating~\eqref{eq:rsecondmoment} in time we obtain
\begin{equation}\label{eq:integratedsecondmoment}
    \int_0^{\infty}\!\!\!\int_{\R^d}|x|^2r(t,a,x)\,dxda = \int_0^{\infty}\!\!\!\int_{\R^d}|x|^2r^0(a,x)\,dxda + \sigma^2\int_0^t\n(s)\,ds.
\end{equation}
Using Corollary~\ref{cor:convol} we deduce
\[\int_0^{\infty}\!\!\!\int_{\R^d}|x|^2r(t,a,x)\,dxda=\int_0^{\infty}\!\!\!\int_{\R^d}|x|^2r^0(a,x)\,dxda + \sigma^2\, Y_1*\n(t)\sim \frac{\sigma^2\int_0^\infty n^0(a)da}{\Psi\Gamma(1-\al)} Y_{1+\al}(t)
\]
from which we get~\eqref{eq:rhomom2} due to the formula
\begin{equation}\label{eq:Gamma1-al1+al}
\Gamma(1-\al)\Gamma(1+\al)=\frac{\pi\al}{\sin(\pi\al)}
\end{equation}
and because $\int_0^\infty n^0(a)da=\int_0^\infty\!\!\int_{\R^d}r^0(a,x)dadx=\int_{\R^d}\rho^0(x)dx$.

\medskip

The computations leading to Equation~\eqref{eq:rsecondmoment} are formal, since there is neither regularity in time nor limit when $a\to\infty$ for $r$ in general.
A less direct but more rigorous way to derive~\eqref{eq:integratedsecondmoment} consists in starting from the integration of Equation~\eqref{eq:jump-renewal} along the characteristics.
Considering the term $\beta(a)r(t,a,x)$ as a (negative) source term, we get
\begin{align*}
    r(t,a,x)=\bigg(&r^0(a-t,x)-\int_0^t\beta(a-t+s)r(s,a-t+s,x)\,ds\bigg)\1_{a>t}\\
    &+\bigg(\r(t-a,x)-\int_0^a\beta(a')r(t-a+a',a',x)\,da'\bigg)\,\1_{a\leq t}.
\end{align*}
Using the boundary condition we obtain, after multiplication by $|x|^2$ and integration in $x$ and $a$,
\begin{align*}
    \int_0^\infty\!\!\!\int_{\R^d}|x|^2r(t,a,x)\,dxda =& \int_t^\infty\!\!\!\int_{\R^d}|x|^2r^0(a-t,x)\,dxda \\
    & -\int_t^\infty\!\!\!\int_{\R^d}\int_0^t|x|^2\beta(a-t+s)r(s,a-t+s,x)\,dsdxda \\
    & + \int_0^t\!\int_0^\infty\!\!\!\iint_{\R^{2d}}|x|^2w(x-y)\beta(a')r(t-a,a',y)\,dydxda'da\\
    & -\int_0^t\int_{\R^d}\int_0^a|x|^2\beta(a')r(t-a+a',a',x)\,da'dxda \\
    = & \int_t^\infty\!\!\!\int_{\R^d}|x|^2r^0(a-t,x)\,dxda \\
    & -\int_0^t\!\int_t^\infty\!\!\!\int_{\R^d}|x|^2\beta(a-t+s)r(s,a-t+s,x)\,dxdads \\
    & + \int_0^\infty\!\!\!\int_0^t\!\iint_{\R^{2d}}|x|^2w(x-y)\beta(a')r(t-a,a',y)\,dydxdada'\\
    & -\int_0^t\int_{a'}^t\int_{\R^d}|x|^2\beta(a')r(t-a+a',a',x)\,dxdada'\\
    = & \int_0^\infty\!\!\!\int_{\R^d}|x|^2r^0(a,x)\,dxda \\
    & -\int_0^t\!\int_s^\infty\!\!\!\int_{\R^d}|x|^2\beta(a')r(s,a',x)\,dxda'ds \\
    & + \int_0^\infty\!\!\!\int_0^t\!\iint_{\R^{2d}}|y|^2w(y-x)\beta(a')r(s,a',x)\,dydxdsda'\\
    & -\int_0^t\int_{a'}^t\int_{\R^d}|x|^2\beta(a')r(s,a',x)\,dxdsda'
\end{align*}
where we have used the Fubini-Tonelli theorem in the second equality (remind that all the functions are non-negative) and changes of variables in the third one.
Yet another use of Fubini-Tonelli allows writing the second term as
\begin{align*}
    \int_0^t\!\int_s^\infty\!\!\!\int_{\R^d}&|x|^2\beta(a')r(s,a',x)\,dxda'ds= \\
    &\int_0^t\!\int_0^{a'}\!\int_{\R^d}|x|^2\beta(a')r(s,a',x)\,dxdsda'
+\int_t^\infty\!\!\!\int_0^t\!\int_{\R^d}|x|^2\beta(a')r(s,a',x)\,dxdsda'.
\end{align*}
All together we obtain
\begin{align*}
    \int_0^\infty\!\!\!\int_{\R^d}&|x|^2r(t,a,x)\,dxda = 
\int_0^\infty\!\!\!\int_{\R^d}|x|^2r^0(a,x)\,dxda \\
&+ \int_0^\infty\!\!\!\int_0^t\!\iint_{\R^{2d}}|y|^2w(y-x)\beta(a')r(s,a',x)\,dydxdsda'
- \int_0^\infty\!\!\!\int_0^t\!\int_{\R^d}|x|^2\beta(a')r(s,a',x)\,dxdsda'.
\end{align*}
Using that $\int w=1$ this identity can also be written as
\begin{align*}
    \int_0^\infty\!\!\!\int_{\R^d}|x|^2r(t,a,x)&\,dxda = 
\int_0^\infty\!\!\!\int_{\R^d}|x|^2r^0(a,x)\,dxda \\
&+ \int_0^\infty\!\!\!\int_0^t\!\iint_{\R^{2d}}(|y|^2-|x|^2)w(y-x)\beta(a')r(s,a',x)\,dydxdsda'.
\end{align*}
Using~\eqref{eq:wsigma} we obtain~\eqref{eq:integratedsecondmoment}.
\end{proof}

\medskip

We end this section with the proof of Proposition~\ref{prop:frac-subdiff}.

\begin{proof}[Proof of Proposition~\ref{prop:frac-subdiff}]
Let $\xi\in C^2([0,\infty),[0,\infty))$ such that $\xi(r)=1$ if $0\leq r\leq1$ and $\xi(r)=0$ if $r \geq2$,
and then define $\varphi\in C^2_c(\R^d)$ by $\varphi(x)=\xi(|x|)$

Using $\varphi_n(x)=\varphi(x/n)$ as a test function in~\eqref{eq:sol} and dominated convergence,
we get~\eqref{eq:rhomom0} by passing to the limit $n\to\infty$.

Similarly, considering the test functions $\varphi_n(x)=x_i\varphi(x/n)$, $1\leq i\leq d$, which verify
\[\Delta\varphi_n(x)=2n^{-1}\partial_i\varphi(x/n)+n^{-2}x_i\Delta\varphi(x/n),\]
we obtain~\eqref{eq:rhomom1} by passing to the limit $n\to\infty$ in~\eqref{eq:sol}, still by dominated convergence.

For~\eqref{eq:rhomom2bis} we proceed in the exact same way with the test function $\varphi_n(x)=|x|^2\varphi(x/n)$, which verifies
\[\Delta\varphi_n(x)=2d\varphi(x/n)+4n^{-1}x\cdot\nabla\varphi(x/n)+|x/n|^2\Delta\varphi(x/n).\]
Passing to the limit $n\to\infty$, again by dominated convergence, we get
\begin{align*}
\int_{\R^d}|x|^2\rho(t,dx)&=\int_{\R^d}|x|^2\rho^0(dx)+2dD_\al\bigg(\int_{\R^d}\rho^0(dx)\bigg) Y_\al*Y_1(t)\\
&=\int_{\R^d}|x|^2\rho^0(dx)+\frac{\sigma^2}{\Psi\Gamma(1-\al)}\bigg(\int_{\R^d}\rho^0(dx)\bigg) Y_{1+\al}(t)\\
&=\int_{\R^d}|x|^2\rho^0(dx)+\frac{\sin(\pi\al)}{\pi\al}\frac{\sigma^2}{\Psi}\bigg(\int_{\R^d}\rho^0(dx)\bigg) t^\al,
\end{align*}
where we have used Lemma~\ref{lem:Yalpha} in the second identity and~\eqref{eq:Gamma1-al1+al} in the last one.
\end{proof}

\section{Proof of the main theorem}

In this section, we consider an initial distribution $r^0_\e$ enjoying Assumption~\eqref{as:r0}, and
we will use the standard duality bracket notation
\[\langle h,\varphi\rangle := \int_{\R^d}h(x)\varphi(x)\,dx\]
for $h\in L^1(\R^d)$ and $\varphi\in L^\infty(\R^d)$, or $h\in\mathcal M(\R^d)$ and $\varphi\in C_b(\R^d)$.

\

In view of Definition~\ref{def:sol}, our first aim is to prove that for any $\varphi\in C^2_c(\R^d)$
\begin{equation}\label{eq:conv1}
\langle \rho_\e(t,\cdot)-\rho_\e^0,\varphi\rangle-D_\al\langle Y_\al*\rho_\e(t,\cdot),\Delta\varphi\rangle\to0\qquad\text{as}\ \e\to0
\end{equation}
locally uniformly in $t\in[0,\infty)$.
Then we will derive compactness estimates ensuring the existence of a sequence $(\rho_{\e_n})_{n\in\N}$ which converges in $C([0,T],\mathcal M)$ for all $T>0$ to a limit $\rho\in C([0,\infty),\mathcal M)$.

\

We first introduce an integral Laplace operator $\Delta^\e$ acting on the space variable as
\[\Delta^\e h(x) = \e^{-2}\int_{\R^d}w_\e(z)\big(h(x-z)-h(x)\big)\,dz.\]
Since we do not assume that $w$ is symmetric, we also define the dual operator $\check\Delta^\e$ by
\begin{equation*}
\check\Delta^\e \varphi(x) = \e^{-2}\int_{\R^d}\check w_\e(z)\big(\varphi(x+z)-\varphi(x)\big)\,dz.
\end{equation*}
We thus have
\begin{equation}\label{eq:duality}
\langle\Delta^\e h,\varphi\rangle=\langle h,\check\Delta^\e\varphi\rangle.
\end{equation}
Besides, under Assumptions~\eqref{as:w} and~\eqref{as:w2}, we have for any $\varphi\in C^2_c(\R^d)$ the uniform convergence
\begin{equation}\label{eq:Deltaconv}\check\Delta^\e\varphi\to\frac{\sigma^2}{2d}\Delta\varphi\qquad\text{as}\ \e\to0.
\end{equation}
To prove~\eqref{eq:conv1} one thus has to show that
\begin{equation}\label{eq:conv2}
\langle\rho_\e-\rho_\e^0,\varphi\rangle-\frac{2d}{\sigma^2}D_\al\langle Y_\al*\rho_\e,\check\Delta^\e\varphi\rangle\to0\qquad\text{as}\ \e\to0
\end{equation}
in $C([0,T],\R)$.
We then use the following lemma.

\begin{lemma}\label{lem:rho-rho0}
We have
\begin{equation*}
\rho_\e-\rho_\e^0=\Delta^\e R_\e
\end{equation*}
where
\[R_\e(t,x)=\e^{2-2/\al}\int_0^t \!\int_0^\infty \beta(a) r_\e(s,a,x)\,dads.\]
\end{lemma}

\begin{proof}
Integrating~\eqref{eq:jump-renewal-eps} on $[0,t]$ in time and on $[0,\infty)$ in age and using the boundary condition we get
\begin{align}
\rho_\e(t,x)-\rho_\e^0(x)&=\e^{-2/\al}\int_0^t \r_\e(s,x)ds -\e^{-2/\al}\int_0^t\! \int_0^\infty\beta(a)r_\e(s,a,x)dads\nonumber\\
&=\e^{-2/\al}\int_0^t\!\int_0^\infty\!\!\!\int_{\R^d}\beta(a)w_\e(z) r_\e(s,a,x-z)\,dzdads -\e^{-2/\al}\int_0^t\! \int_0^\infty\beta(a)r_\e(s,a,x)dads\nonumber\\
&=\e^{2-2/\al}\int_0^t \!\int_0^\infty \beta(a)\Delta^\e r_\e(s,a,x)dads.\label{eq:rho_integ}
\end{align}
\end{proof}

Using~\eqref{eq:duality}, one is thus led, to prove~\eqref{eq:conv2}, to estimate
\[R_\e-\frac{2dD_\al}{\sigma^2}\,Y_\al*\rho_\e.\]
To do so, we separate the contributions of jumpers (that have jumped before time $t$) and non-jumpers (that have only aged).
Integrating Equation~\eqref{eq:jump-renewal-eps} along the characteristics one gets
\begin{align*}
 r_\e(t,a,x)&=r_\e(t-\e^{2/\alpha}a,0,x)e^{-\int_0^a\beta(a')\,da'}\,\1_{t\geq\e^{2/\alpha}a}\\
  &\qquad\qquad+r_\e(0,a-\e^{-2/\alpha}t,x)e^{-\int_0^{\e^{-2/\alpha}t}\beta(a'+a-\e^{-2/\alpha}t)\,da'}\,\1_{t<\e^{2/\alpha}a},
  \end{align*}
  which also reads
\begin{equation*}
r_\e(t,a,x)=\r_\e(t-\e^{2/\alpha}a,x)\psi(a)\,\1_{t\geq\e^{2/\alpha}a}+r_\e^0(a-\e^{-2/\alpha}t,x)\f{\psi(a)}{\psi(a-\e^{-2/\alpha}t)}\,\1_{t<\e^{2/\alpha}a}.
\end{equation*}
This allows us to split $\rho_\e$ as
\begin{equation}\label{eq:split_rho}
\rho_\epsilon(t,x)=\rho_\epsilon^{j}(t,x)+\rho_\epsilon^{nj}(t,x),
\end{equation}
where
\[\rho_\epsilon^{j}(t,x) = \int_0^{\epsilon^{-2/\alpha}t} r_\epsilon(t,a,x)\,da = \e^{-2/\alpha} \int_0^{t}\psi(\epsilon^{-2/\alpha}(t-s))\r_\epsilon(s,x) ds\]
is the contribution of jumpers, and
\[\rho_\epsilon^{nj}(t,x) = \int_{\epsilon^{-2/\alpha}t}^\infty r_\epsilon(t,a,x)\,da = \int_0^\infty\f{\psi(a+\epsilon^{-2/\alpha}t)}{\psi\left(a\right)}r_\e^0(a,x) da\]
the one of non-jumpers.
Similarly, one can write
\begin{equation}\label{eq:split_R}
R_\e(t,x)=R_\e^{j}(t,x) + R_\e^{nj}(t,x)
\end{equation}
with
\begin{align*}
R_\e^j(t,x) & = \e^{2-2/\al}\int_0^t \int_0^{\e^{-2/\al}s} \beta(a)r_\e(s,a,x) \,da  ds\\
& = \e^{2-4/\al}\int_0^t \int_0^{s} \phi(\e^{-2/\al}(s-s'))\r_\e(s',x) \,ds'  ds
\end{align*}
and
\begin{align*}
R_\e^{nj}(t,x) & = \e^{2-2/\al}\int_0^t \int_{\e^{-2/\al}s}^\infty \ \beta(a)r_\e(s,a,x) \,da  ds\\
& = \e^{2-2/\al}\int_0^t \int_0^\infty \frac{\phi(a+\e^{-2/\al}s)}{\psi(a)}r_\e^0(a,x) \,da ds .
\end{align*}
We will then estimate separately
\[R_\e^j-\frac{2dD_\al}{\sigma^2}\,Y_\al*\rho_\e^j\qquad\text{and}\qquad R_\e^{nj}-\frac{2dD_\al}{\sigma^2}\,Y_\al*\rho_\e^{nj}.\]
We start with the jumpers.
\begin{lemma}\label{lem:jump}
Under Assumption~\eqref{as:beta}, one has for all $T>0$
\[\sup_{0\leq t\leq T}\int_{\R^d}\Big|R_\e^j- \frac{2dD_\al}{\sigma^2} Y_\alpha*\rho_\e^j\Big|(t,x)\,dx=O(\e^2)+O(\e^{2\delta/\al})\qquad\text{as}\ \e\to0.\]
\end{lemma}

\begin{proof}
We start by proving that
\begin{equation}\label{eq:jumpers}
\Big(R_\eps^j- \frac{2dD_\al}{\sigma^2} Y_\alpha*\rho_\eps^j\Big)(t,x)= \eps^2 \, \r_\eps(\eps^{2/\alpha}\,\cdot\,,x)*\Big(1-\psi-\frac{2dD_\al}{\sigma^2} Y_\alpha*\psi\Big)(\eps^{-2/\alpha}t).
\end{equation}
On the one hand we have
\begin{align*}
R_\epsilon^j(t,x)&= \epsilon^{2-4/\alpha}\int_0^t \int_0^{s} \phi(\eps^{-2/\alpha}(s-s')) \r_\epsilon(s',x)ds'ds
\\
&= \epsilon^{2-4/\alpha}\int_0^t \r_\epsilon(s',x)\int_{s'}^{t} \phi(\eps^{-2/\alpha}(s-s')) dsds'
\\
&= \epsilon^{2-4/\alpha}\int_0^t \r_\epsilon(s',x)\int_0^{t-s'} \phi(\eps^{-2/\alpha}(s)) dsds'
\\
&=\epsilon^{2-2/\alpha}\int_0^t \r_\epsilon(s',x)(1-\psi(\eps^{-2/\alpha}(t-s'))ds'.
\end{align*}
On the other hand we have
\begin{align*}
Y_\alpha*\rho_\eps^j(t,x) &=\frac{\eps^{-2/\al}}{\Gamma(\al)}\int_0^t (t-s)^{\al-1}\int_0^s\psi(\eps^{-2/\alpha}(s-s'))\r_\eps(s',x)\,ds'ds\\
&=\frac{\eps^{-2/\al}}{\Gamma(\al)}\int_0^t\r_\eps(s',x)\int_{s'}^t (t-s)^{\alpha-1}\psi(\eps^{-2/\alpha}(s-s'))\,dsds'\\
&=\frac{\eps^{-2/\al}}{\Gamma(\al)}\int_0^t\r_\eps(s',x)\int_0^{t-s'} (t-s'-s)^{\alpha-1}\psi\big(\e^{-2/\alpha}s\big)\,dsds'\\
 &=\frac{1}{\Gamma(\al)}\int_0^t\r_\eps(s',x)\int_0^{\eps^{-2/\alpha}(t-s')} (t-s'-\eps^{2/\alpha}s)^{\alpha-1}\psi(s)\,dsds'\\
 &=\frac{1}{\Gamma(\al)}\int_0^t\r_\eps(s',x)\int_0^{\eps^{-2/\alpha}(t-s')} (\eps^{2/\alpha})^{\alpha-1} (\eps^{-2/\alpha}(t-s')-s)^{\alpha-1}\psi(s)\,dsds'\\
 &=\eps^{2-2/\alpha}\int_0^t\r_\eps(s',x)\,Y_\alpha*\psi(\e^{-2/\alpha}(t-s'))ds'.
\end{align*}
This gives
\begin{align*}
\Big(R_\eps^j- \frac{2dD_\al}{\sigma^2} Y_\alpha*\rho_\eps^j\Big)(t,x) &= \eps^{2-2/\alpha}\int_0^t \r_\eps(s,x) \left(1-\psi(\eps^{-2/\alpha}(t-s))-\frac{2dD_\al}{\sigma^2}\,Y_\alpha*\psi (\eps^{-2/\alpha}(t-s))\right)ds\\
&= \eps^{2}\int_0^{\eps^{-2/\alpha}t} \r_\eps(\eps^{2/\alpha}s,x) \left(1-\psi(\eps^{-2/\alpha}t-s)-\frac{2dD_\al}{\sigma^2}\,Y_\alpha*\psi (\eps^{-2/\alpha}t-s)\right)ds,
\end{align*}
which is exactly~\eqref{eq:jumpers}.
Now we use that~\eqref{eq:Ynupsi} with $\nu=\al$ gives
\[Y_\alpha*\psi(t) = \Psi \Gamma(1-\alpha) + O(t^{-\delta}).\]
Recalling that $D_\al=\frac{\sigma^2}{2d\Psi\Gamma(1-\al)}$, we deduce the existence of $C>0$ such that
\begin{equation}\label{eq:Odelta}
\Big|1-\frac{2dD_\al}{\sigma^2}\,Y_\al*\psi(t)\Big|\leq CY_{1-\delta}(t).
\end{equation}
We now define the function $n_\e$ by
\[n_\e(t,a)=\int_{\R^d}r_\e(\e^{2/\al}t,a,x)\,dx.\]
Since
\[\n_\ep(t)=n_\e(t,0)=\int_{\R^d}\r_\e(\e^{2/\al}t,x)\,dx,\]
one obtains by integration of~\eqref{eq:jumpers} in $x$ and using~\eqref{eq:Odelta} that
\begin{align}\label{eq:Rj-estim}
    \int_{\R^d}\Big|R_\epsilon^j- \frac{2dD_\al}{\Gamma(\al)} Y_\alpha*\rho_\eps^j\Big|(t,x)\,dx
    & \leq \eps^2 \, \n_\e*\Big|1-\psi-\frac{2dD_\al}{\Gamma(\al)} Y_\alpha*\psi\Big|(\eps^{-2/\alpha}t) \nonumber\\
    & \leq \eps^2\, \n_\e*\psi(\eps^{-2/\alpha}t)+\e^2\,\n_\e*\Big|1-\frac{2dD_\al}{\Gamma(\al)} Y_\alpha*\psi\Big|(\eps^{-2/\alpha}t) \nonumber\\
    &\leq \eps^2\, \n_\e*\psi(\eps^{-2/\alpha}t)+C\e^2\,\n_\e*Y_{1-\delta}(\eps^{-2/\alpha}t).
\end{align}
Because the function $n_\e$ verifies Equation~\eqref{eq:age_density} with initial condition
\[n_\e^0(a) = \int_{\R^d}r_\e^0(a,x)\,dx,\]
which enjoys the convergence
\[\int_0^\infty n_\e^0(a) \,da = \int_{\R^d} \rho_\e^0(x)\,dx\xrightarrow[\e\to0]{}\int_{\R^d}\rho^0(x)\,dx,\]
one can use Lemma~\ref{lem:renewal} and Corollary~\ref{cor:convol} to infer from~\eqref{eq:Rj-estim} that
\[\int_{\R^d}\Big|R_\epsilon^j- \frac{2dD_\al}{\Gamma(\al)} Y_\alpha*\rho_\eps^j\Big|(t,x)\,dx=O(\e^2)+O(\e^{2\delta/\al})\]
locally uniformly in $t\in[0,\infty)$.
\end{proof}

We now turn to the non-jump part.

\begin{lemma}\label{lem:nonjump}
Assume that~\eqref{as:beta} holds.
Then for all $T>0$
\[\sup_{0\leq t\leq T}\int_{\R^d}\Big|R_\e^{nj}- \frac{2dD_\al}{\sigma^2} Y_\alpha*\rho_\e^{nj}\Big|(t,x)\,dx=O(\e^2)\qquad\text{as}\ \e\to0.\]
\end{lemma}

\begin{proof}
We start by computing
\begin{align*}
\Big(R_\e^{nj}- \frac{2D_\al}{\sigma^2} Y_\alpha*\rho_\e^{nj}\Big)(t,x) & = \epsilon^{2-2/\al}\int_0^t \int_{0}^\infty \frac{\phi(a+\epsilon^{-2/\al}s)}{\psi(a)} r_\epsilon^0(a,x)dads\\
&\qquad\qquad -\int_0^t\frac{(t-s)^{\al-1}}{\Gamma(\alpha)}\int_0^\infty \frac{\psi(a+\e^{-2/\al}s)}{\psi(a)}r_\e^0(a,x)da ds\\
&= \epsilon^{2}\int_{0}^\infty \frac{\psi(a)-\psi(a+\e^{-2/\al}t)}{\psi(a)} r_\epsilon^0(a,x)da\\
&\qquad\qquad -\e^2\int_0^{\e^{-2/\al}t}\frac{(\e^{-2/\al}t-s)^{\al-1}}{\Gamma(\al)}\int_0^\infty \frac{\psi(a+s)}{\psi(a)}r_\e^0(a,x)da ds
\\
&=\epsilon^2
\int_0^\infty r_\e^0(a,x)\frac{\psi(a)-\psi(a+\e^{-2/\alpha} t)-(Y_\al*\psi(a+\cdot)) (\e^{-2/\al}t)}{\psi(a)}da.
\end{align*}
To estimate this integral, we use that $\psi$ is a non-increasing function to get
\[0\leq\frac{\psi(a)-\psi(a+\e^{-2/\alpha} t)}{\psi(a)}\leq1\]
and
\[(Y_\al*\psi(a+\cdot)) (\e^{-2/\al}t)\leq Y_\al*\psi (\e^{-2/\al}t)\leq C\, Y_\al*Y_{1-\al}(\e^{-2/\al}t) = C,\]
where we have used Assumption~\eqref{as:beta} to ensure the existence of $C>0$ such $\psi\leq CY_{1-\al}$ in the second inequality, and Lemma~\ref{lem:Yalpha} for the last equality.
The result then follows from the fact $r^0_\e$ verifies the condition~\eqref{as:r0} and the fact that $(1+a)^\al\psi(a)$ is bounded on $[0,\infty)$, again as a consequence of~\eqref{as:beta}.
\end{proof}

As a consequence of these lemmas, one obtains a quantified version of~\eqref{eq:conv2}.

\begin{corollary}\label{cor:conv2}
Under Assumptions~\eqref{as:w},~\eqref{as:w2} and~\eqref{as:beta}, we have for all $\varphi\in C^2_c(\R^d)$ and all $T>0$
\[\sup_{0\leq t\leq T}\Big|\langle\rho_\e(t,\cdot)-\rho_\e^0,\varphi\rangle-\frac{2dD_\al}{\sigma^2}\langle Y_\al*\rho_\e(t,\cdot),\check\Delta^\e\varphi\rangle\Big| = O(\e^2)+O(\e^{2\delta/\al})\qquad\text{as}\ \e\to0.\]
\end{corollary}

\begin{proof}
Using Lemma~\ref{lem:rho-rho0} one has for all $t\geq0$
\begin{align*}
\Big|\langle\rho_\e(t,\cdot)-\rho_\e^0,\varphi\rangle-\frac{2dD_\al}{\sigma^2}\langle Y_\al*\rho_\e(t,\cdot),\check\Delta^\e\varphi\rangle\Big| & = \Big|\langle\Delta^\e R_\e(t,\cdot),\varphi\rangle-\frac{2dD_\al}{\sigma^2}\langle Y_\al*\rho_\e(t,\cdot),\check\Delta^\e\varphi\rangle\Big| \\
&\leq \big\|\check\Delta^\e\varphi\big\|_\infty\int_{\R^d}\Big|R_\e(t,x)- \frac{2dD_\al}{\sigma^2} Y_\alpha*\rho_\e(t,x)\Big|\,dx.
\end{align*}
The conclusion then follows from~\eqref{eq:Deltaconv} and Lemmas~\ref{lem:jump} and~\ref{lem:nonjump}, by using the splittings~\eqref{eq:split_rho} and~\eqref{eq:split_R}.
\end{proof}

As announced, the second step is to prove a time equicontinuity result on the family $\{\rho_\e,\ 0<\e<1\}$.

\begin{lemma}\label{lem:equicontinuity}
Assume~\eqref{as:w},~\eqref{as:w2} and~\eqref{as:beta}.
Then for all $\varphi\in C^2_c(\R^d)$ there exists a function $\varpi\in C([0,\infty)^2,[0,\infty))$ satisfying $\varpi(t,t)=0$ for any $t\geq0$ and: for all $T>0$ there is a constant $C>0$ such that
\[|\langle\rho_\e(t,\cdot),\varphi\rangle-\langle\rho_\e(t',\cdot),\varphi\rangle|\leq\varpi(t,t')+C\big(\epsilon^{2}+\e^{2\delta/\al}\big)\]
 for all $t,t'\in[0,T]$ and all $\e\in(0,1)$.
\end{lemma}

\begin{proof}
Due to Corollary~\ref{cor:conv2}, we only have to prove that
\[|\langle Y_\al*\rho_\e(t,\cdot),\check\Delta^\e\varphi\rangle-\langle Y_\al*\rho_\e(t',\cdot),\check\Delta^\e\varphi\rangle|\leq \omega(t,t')\]
for all $t,t'\in[0,T]$ and $\e\in(0,1)$.
From~\eqref{eq:rhomom1} we have that for any $t\geq0$
\[|\langle\rho_\e(t,\cdot),\check\Delta^\e\varphi\rangle|\leq{\|\rho_\e(t,\cdot)\|}_{L^1}{\|\check\Delta^\e\varphi\|}_\infty={\|\rho_\e^0\|}_{L^1}{\|\check\Delta^\e\varphi\|}_\infty\]
and we deduce that for $t'\geq t\geq0$
\begin{align*}
|Y_\al*\langle\rho_\e,\check\Delta^\e\varphi\rangle(t)-Y_\al*&\langle\rho_\e,\check\Delta^\e\varphi\rangle(t')|=\frac{1}{\Gamma(\al)}\bigg|\int_t^{t'}(t'-s)^{\al-1}\langle\rho_\e,\check\Delta^\e\varphi\rangle(s)\, ds\\
&\hskip44mm-\int_0^t\big((t-s)^{\al-1}-(t'-s)^{\al-1}\big)\langle\rho_\e,\check\Delta^\e\varphi\rangle(s)\, ds\bigg| \\
&\leq \frac{{\|\rho_\e^0\|}_{L^1}{\|\check\Delta^\e\varphi\|}_\infty}{\Gamma(\al)}\bigg(\int_t^{t'}(t'-s)^{\al-1} ds+\int_0^t\big((t-s)^{\al-1}-(t'-s)^{\al-1}\big) ds\bigg) \\
&\leq \frac{{\|\rho_\e^0\|}_{L^1}{\|\check\Delta^\e\varphi\|}_\infty}{\al\Gamma(\al)} \big(2(t'-t)^\al+t^\al-t'^\al\big).
\end{align*}
The result follows from the fact that
\[{\|\rho_\e^0\|}_{L^1}\to{\|\rho^0\|}_{L^1}\quad\text{and}\quad{\|\check\Delta^\e\varphi\|}_\infty\to{\|\Delta\varphi\|}_\infty \qquad\text{as}\ \e\to0.\]
\end{proof}

\begin{corollary}\label{cor:equicont}
For any $\varphi\in C_c^2(\R^d)$, there exists a constant $C$ depending only on \(\alpha,\delta,\|\beta\|_\infty,\|\rho^0\|_1,\|\varphi\|_{W^{2,\infty}}\) such that 
\[
\left|\langle \rho_\ep(t,.),\varphi\rangle -\langle \rho_\ep(t',.),\varphi\rangle  \right|\leq C\left(2|t-t'|^\al +|t^\al-t'^\al|+|t-t'|^{\al}+|t-t'|^{\f{\delta}{1-\al+\delta}}\right)
\]
\end{corollary}
\begin{proof}
From \eqref{eq:rho_integ}, we have for $t,t'$,
\[
\rho_\ep(t,x)-\rho(t',x)=\ep^{2-\f{2}{\al}}\int_{t'}^t \int_0^a \beta(a) \Delta^\ep r_\ep(s,a,x) da ds.
\]
This leads to, for $\varphi\in C^2_c$, 
\begin{align*}
    \langle \rho_\ep(t,.),\varphi\rangle -\langle \rho_\ep(t',.),\varphi\rangle &= \ep^{2-\f{2}{\al}}\int_{t'}^t \int_0^\infty \beta(a)  r_\ep(s,a,x) \check\Delta^\ep\varphi (x) da ds.
\end{align*}
Leading to (suboptimal estimate) 
\begin{align*}
   \left| \langle \rho_\ep(t,.),\varphi\rangle -\langle \rho_\ep(t,.),\varphi\rangle \right| &\leq  \ep^{2-\f{2}{\al}}|t-t'| \|\beta\|_\infty \|r_\ep\|_{L^1(dadx)}\| \check\Delta^\ep\varphi\|_\infty \\
   &= \ep^{2-\f{2}{\al}}|t-t'| \|\beta\|_\infty \|\rho_\ep^0\|_{1}\| \check\Delta^\ep\varphi\|_\infty
\end{align*}
Combining everything together, we obtain set of two majorations that we can merge into 
\begin{align*}
   \left| \langle \rho_\ep(t,.),\varphi\rangle -\langle \rho_\ep(t,.),\varphi\rangle \right|\leq C\min\left( 2|t-t'|^\al +|t^\al-t'^\al|+(\ep^2+\ep^{\f{2\delta}{\al}}),\ep^{2-\f{2}{\al}}|t-t'|\right)
   \end{align*}
   where $C$ depends on $\beta,\|\rho_1^0\|, \|\varphi\|_{W^{2,\infty}},\al$ but neither on $\ep$ or $t,t'$. This leads to \begin{align*}
   \left| \langle \rho_\ep(t,.),\varphi\rangle -\langle \rho_\ep(t,.),\varphi\rangle \right|\leq& C\left( 2|t-t'|^\al +|t^\al-t'^\al|\right)\\
   &+C\min\left(\ep^2,\ep^{2-\f{2}{\al}}|t-t'|\right)+C\min\left(\ep^{\f{2\delta}{\al}},\ep^{2-\f{2}{\al}}|t-t'|\right)
   \end{align*}
   Using the elementary inequality 
   \[
   \forall a,b>0, \forall \gamma \in [0,1],\quad \min(a,b)\leq a^\gamma b^{1-\gamma}
   \]
   and choosing the values of $\gamma=1-\al, \frac{1-\al}{\delta+1-\al}$ to get rid of $\ep$, we get to 
   \begin{equation}
      \left| \langle \rho_\ep(t,.),\varphi\rangle -\langle \rho_\ep(t,.),\varphi\rangle \right|
 \leq C\left(2|t-t'|^\al +|t^\al-t'^\al|+|t-t'|^{\al}+|t-t'|^{\f{\delta}{1-\al+\delta}}\right) \label{eq:equicont}
   \end{equation}
 where again the constant $C$ depend on $\varphi,\beta,\|\rho^0\|_1,\al$ but not on $\ep$.
\end{proof}

We are now in position to prove our main theorem.

\begin{proof}[Proof of Theorem~\ref{th:main}]
Because of Corollary~\ref{cor:equicont}, one can invoke Arzelà-Ascoli Theorem to infer, for any $T>0$ and any $\varphi\in C^2_c(\R^d)$, the existence of a continuous fonction $\ell_\varphi:[0,T]\to\R$ and a sequence $(\e_n)$ decreasing to zero such that

\begin{equation}\label{eq:conv_lphi}
\sup_{0\leq t\leq T}\big|\langle\rho_{\e_n}(t,\cdot),\varphi\rangle-\ell_\varphi(t)\big|\to0.
\end{equation}
Let us pick up a sequence ${(\varphi_k)}_{k\in\N}$ of $C^2_c(\R^d)$ functions which is dense in $C_0(\R^d)$.
By a diagonal argument, we can extract a new subsequence, still denoted by $(\e_n)$, such that
\eqref{eq:conv_lphi} holds for any $\varphi\in\{\varphi_k,\,k\in\N\}$.
Now, for any fixed $t\in[0,T]$, we define a linear form on $\{\varphi_k,\,k\in\N\}$ by
\[L_t(\varphi_k):=\ell_{\varphi_k}(t).\]
Using~\eqref{eq:rhomom0} and~\eqref{as:r0}, we have for all $k\in\N$
\[|L_t(\varphi_k)|\leq \Big(\limsup_{n\to\infty}\int_{\R^d}\rho_{\e_n}(t,x)dx\Big){\|\varphi_k\|}_\infty
=\Big(\limsup_{n\to\infty}\int_{\R^d}\rho^0_{\e_n}(x)dx\Big){\|\varphi_k\|}_\infty\leq M{\|\varphi_k\|}_\infty.\]
The linear form $L_t$ is thus continuous and we can extend it into a unique continuous linear form with norm less than $M$, still denoted by $L_t$, on $C_0(\R^d)$.
Riesz's representation theorem then provides us with the existence of a measure $\rho(t)\in\mathcal M(\R^d)$ such that
\[\forall\varphi\in C_0(\R^d),\qquad L_t(\varphi)=\langle\rho(t),\varphi\rangle.\]
We now verify that $\rho$ is the weak$*$ limit of the sequence $(\rho_{\e_n})$, uniformly in time on $[0,T]$.
For any $\varphi\in C_0(\R^d)$ and any $k\in\N$ we have
\begin{align*}
\sup_{0\leq t\leq T}\big|\langle\rho_{\e_n}(t,\cdot),\varphi\rangle-\langle\rho(t),\varphi\rangle\big|&
\leq \sup_{0\leq t\leq T}\big|\langle\rho_{\e_n}(t,\cdot),\varphi_k\rangle-\langle\rho(t),\varphi_k\rangle\big|\\
&\qquad+\sup_{0\leq t\leq T}\big|\langle\rho_{\e_n}(t,\cdot),\varphi_k-\varphi\rangle\big|
+\sup_{0\leq t\leq T}\big|\langle\rho(t),\varphi_k-\varphi\rangle\big|\\
&\leq \sup_{0\leq t\leq T}\big|\langle\rho_{\e_n}(t,\cdot),\varphi_k\rangle-\langle\rho(t),\varphi_k\rangle\big|
+2M{\|\varphi_k-\varphi\|}_\infty.
\end{align*}
Since $\langle\rho(t),\varphi_k\rangle=L_t(\varphi_k)=\ell_{\varphi_k}(t)$, the convergence~\eqref{eq:conv_lphi}, which is valid for $\varphi_k$, ensures that
\[\limsup_{n\to\infty}\sup_{0\leq t\leq T}\big|\langle\rho_{\e_n}(t,\cdot),\varphi\rangle-\langle\rho(t),\varphi\rangle\big| \leq 2M{\|\varphi_k-\varphi\|}_\infty.\]
The density of $(\varphi_k)$ in $C_0(\R^d)$ guaranteeing that $\inf_k{\|\varphi_k-\varphi\|}_\infty=0$, we obtain that
\[\sup_{0\leq t\leq T}\big|\langle\rho_{\e_n}(t,\cdot),\varphi\rangle-\langle\rho(t),\varphi\rangle\big|\to0\]
which, recalling that we have endowed the space $\mathcal M(\R^d)$ with its weak$*$ topology, exactly means that
\begin{equation}\label{eq:conv_rho}
\rho_{\e_n}\to \rho\qquad\text{in}\ C([0,T],\mathcal M(\R^d)).
\end{equation}
Finally, considering a sequence of times $T\in\N$, we infer, again by a diagonal argument, the existence of a subsequence, still denoted by $(\e_n)$, and the existence of $\rho\in C([0,\infty),\mathcal M(\R^d))$ such that the convergence~\eqref{eq:conv_rho} holds for all $T>0$.
Since Corollary~\ref{cor:conv2} guarantees that~\eqref{eq:conv1} holds for any $\varphi\in C^2_c(\R^d)$,
we get by passing to the limit that $\rho$ must verify Definition~\ref{def:sol} of a solution to Equation~\eqref{eq:frac-subdiff}.
\end{proof}

\section*{Acknowledgement}
The authors acknowledge the support of the French Agency for Research, ANR (grant ANR-20-CE45-0023, ABC4M) to H.B. and N.Q. 


\begin{thebibliography}{10}

\bibitem{amblard_subdiffusion_1996}
F.~Amblard, A.~C. Maggs, B.~Yurke, A.~N. Pargellis, and S.~Leibler.
\newblock Subdiffusion and {Anomalous} {Local} {Viscoelasticity} in {Actin}
  {Networks}.
\newblock {\em Physical Review Letters}, 77(21):4470--4473, 1996.

\bibitem{Bernard2010}
E.~Bernard, E.~Caglioti, and F.~Golse.
\newblock Homogenization of the linear {Boltzmann} equation in a domain with a
  periodic distribution of holes.
\newblock {\em SIAM J. Math. Anal.}, 42(5):2082--2113, 2010.

\bibitem{Berry2016}
H.~Berry, T.~Lepoutre, and {\'A}.~M. Gonz{\'a}lez.
\newblock Quantitative convergence towards a self-similar profile in an
  age-structured renewal equation for subdiffusion.
\newblock {\em Acta Appl. Math.}, 145(1):15--45, 2016.

\bibitem{bressloff2022stochastic}
P.~C. Bressloff.
\newblock {\em Stochastic Processes in Cell Biology: Volume I}, volume~41.
\newblock Springer Nature, 2022.

\bibitem{bressloff2013stochastic}
P.~C. Bressloff and J.~M. Newby.
\newblock Stochastic models of intracellular transport.
\newblock {\em Reviews of Modern Physics}, 85(1):135--196, 2013.

\bibitem{Calvez2019}
V.~Calvez, P.~Gabriel, and {\'A}.~M. Gonz{\'a}lez.
\newblock Limiting {Hamilton}-{Jacobi} equation for the large scale asymptotics
  of a subdiffusion jump-renewal equation.
\newblock {\em Asymptotic Anal.}, 115(1-2):63--94, 2019.

\bibitem{caspi_enhanced_2000}
A.~Caspi, R.~Granek, and M.~Elbaum.
\newblock Enhanced {Diffusion} in {Active} {Intracellular} {Transport}.
\newblock {\em Physical Review Letters}, 85(26):5655--5658, 2000.

\bibitem{DongKim}
H.~Dong and D.~Kim.
\newblock {{\(L_p\)}}-estimates for time fractional parabolic equations in
  divergence form with measurable coefficients.
\newblock {\em J. Funct. Anal.}, 278(3):66, 2020.
\newblock Id/No 108338.

\bibitem{evangelista2018fractional}
L.~R. Evangelista and E.~K. Lenzi.
\newblock {\em Fractional diffusion equations and anomalous diffusion}.
\newblock Cambridge University Press, 2018.

\bibitem{evangelista2023introduction}
L.~R. Evangelista and E.~K. Lenzi.
\newblock {\em An introduction to anomalous diffusion and relaxation}.
\newblock Springer, 2023.

\bibitem{Gabriel2018}
P.~Gabriel.
\newblock Measure solutions to the conservative renewal equation.
\newblock {\em ESAIM, Proc. Surv.}, 62:68--78, 2018.

\bibitem{Gorenflo2008}
R.~Gorenflo and F.~Mainardi.
\newblock {\em {C}ontinuous {T}ime {R}andom {W}alk, {M}ittag-{L}effler
  {W}aiting {T}ime and {F}ractional {D}iffusion: {M}athematical {A}spects},
  chapter~4, pages 93--127.
\newblock John Wiley \& Sons, Ltd, 2008.

\bibitem{Graczyk2023}
K.~Graczyk, D.~Strzelczyk, and M.~Matyka.
\newblock Deep learning for diffusion in porous media.
\newblock {\em Scientific Reports}, 13, 06 2023.

\bibitem{Haugh2009}
J.~M. Haugh.
\newblock Analysis of reaction-diffusion systems with anomalous subdiffusion.
\newblock {\em Biophysical journal}, 97 2:435--42, 2009.

\bibitem{hofling2013}
F.~Höfling and T.~Franosch.
\newblock Anomalous transport in the crowded world of biological cells.
\newblock {\em Reports on Progress in Physics}, 76(4):046602, 2013.

\bibitem{izeddin_single-molecule_2014}
I.~Izeddin, V.~Récamier, L.~Bosanac, I.~I. Cissé, L.~Boudarene,
  C.~Dugast-Darzacq, F.~Proux, O.~Bénichou, R.~Voituriez, O.~Bensaude,
  M.~Dahan, and X.~Darzacq.
\newblock Single-molecule tracking in live cells reveals distinct target-search
  strategies of transcription factors in the nucleus.
\newblock {\em eLife}, 3:e02230, 2014.

\bibitem{Zacher2016}
J.~Kemppainen, J.~Siljander, V.~Vergara, and R.~Zacher.
\newblock Decay estimates for time-fractional and other non-local in time
  subdiffusion equations in {$\Bbb{R}^d$}.
\newblock {\em Math. Ann.}, 366(3-4):941--979, 2016.

\bibitem{Metzler2020}
T.~Koszto\l{}owicz and R.~Metzler.
\newblock Diffusion of antibiotics through a biofilm in the presence of
  diffusion and absorption barriers.
\newblock {\em Phys. Rev. E}, 102:032408, Sep 2020.

\bibitem{Lee2021}
D.~Lee, N.~Wingreen, and C.~Brangwynne.
\newblock Chromatin mechanics dictates subdiffusion and coarsening dynamics of
  embedded condensates.
\newblock {\em Nature Physics}, 17:1--8, 04 2021.

\bibitem{Mendez2010}
V.~M\'endez, S.~Fedotov, and W.~Horsthemke.
\newblock {\em Reaction-transport systems}.
\newblock Springer Series in Synergetics. Springer, Heidelberg, 2010.
\newblock Mesoscopic foundations, fronts, and spatial instabilities.

\bibitem{Metzler2000}
R.~Metzler and J.~Klafter.
\newblock The random walk's guide to anomalous diffusion: {A} fractional
  dynamics approach.
\newblock {\em Phys. Rep.}, 339(1):1--77, 2000.

\bibitem{MontrollWeiss1965}
E.~W. Montroll and G.~H. Weiss.
\newblock Random walks on lattices. {II}.
\newblock {\em Journal of Mathematical Physics}, 6(2):167--181, 1965.

\bibitem{perthame2007transport}
B.~Perthame.
\newblock {\em Transport equations in biology.}
\newblock Front. Math. Birkh\"auser Verlag, Basel, 2007.

\bibitem{Perthame2025}
B.~Perthame and M.~Tang.
\newblock Deriving sub-diffusion equations.
\newblock {\em preprint arXiv:2507.20602}, 2025.

\bibitem{Saxton2007}
M.~Saxton.
\newblock A biological interpretation of transient anomalous subdiffusion. I.
  qualitative model.
\newblock {\em Biophys. J.}, 92(4):1178--1191, Feb 2007.

\bibitem{Scalas2004}
E.~Scalas, R.~Gorenflo, and F.~Mainardi.
\newblock Uncoupled continuous-time random walks: Solution and limiting
  behavior of the master equation.
\newblock {\em Phys. Rev. E}, 69(1):011107, 2004.

\bibitem{suzuki2023fractional}
J.~L. Suzuki, M.~Gulian, M.~Zayernouri, and M.~D’Elia.
\newblock Fractional modeling in action: A survey of nonlocal models for
  subsurface transport, turbulent flows, and anomalous materials.
\newblock {\em Journal of Peridynamics and Nonlocal modeling}, 5(3):392--459,
  2023.

\bibitem{Vlad1984}
M.~Vlad, V.~Popa, and E.~Segal.
\newblock A master equation description of the age dependent systems.
\newblock {\em Phys. Lett. A}, 100(8):387--391, 1984.

\bibitem{Vlad1987}
M.~O. Vlad.
\newblock Age distributions in physical systems.
\newblock {\em J. Phys. A}, 20(11):3367, 1987.

\bibitem{Vlad2002}
M.~O. Vlad and J.~Ross.
\newblock Systematic derivation of reaction-diffusion equations with
  distributed delays and relations to fractional reaction-diffusion equations
  and hyperbolic transport equations: Application to the theory of neolithic
  transition.
\newblock {\em Phys. Rev. E}, 66:061908, 2002.

\bibitem{wang2022anomalous}
W.~Wang, R.~Metzler, and A.~G. Cherstvy.
\newblock Anomalous diffusion, aging, and nonergodicity of scaled brownian
  motion with fractional gaussian noise: Overview of related experimental
  observations and models.
\newblock {\em Physical Chemistry Chemical Physics}, 24(31):18482--18504, 2022.

\bibitem{Yadav2006}
A.~Yadav and W.~Horsthemke.
\newblock Kinetic equations for reaction-subdiffusion systems: Derivation and
  stability analysis.
\newblock {\em Phys. Rev. E}, 74(6):066118, 2006.

\bibitem{Zacher2013}
R.~Zacher.
\newblock A {De} {Giorgi}-{Nash} type theorem for time fractional diffusion
  equations.
\newblock {\em Math. Ann.}, 356(1):99--146, 2013.

\end{thebibliography}

\end{document}